\documentclass[10pt]{amsart}

\usepackage{amssymb,MnSymbol}
\usepackage{amsthm,amsmath,enumerate,color}
\usepackage{cite}
\usepackage[utf8]{inputenc}
\usepackage{graphicx}

\theoremstyle{plain}
\newtheorem{theorem}{Theorem}[section]
\newtheorem{lemma}[theorem]{Lemma}
\newtheorem{proposition}[theorem]{Proposition}
\newtheorem{corollary}[theorem]{Corollary}

\newtheorem{observation}[theorem]{Observation}

\theoremstyle{definition}
\newtheorem{de}[theorem]{Definition}
\newtheorem{example}[theorem]{Example}
\newtheorem{conj}[theorem]{Conjecture}

\theoremstyle{remark}
\newtheorem{remark}{Remark}

\def\stb{,\ldots ,}

\begin{document}
\title{Associative idempotent nondecreasing functions are reducible}
\author{Gergely Kiss and G\'abor Somlai }

\keywords{$n$-ary semigroup, associativity, reducible, extremal,
quasitrivial, idempotent, neutral element}

\subjclass[2010]{Primary 20N15, 39B72; Secondary 20M14, 20M30,
06F05}
\thanks{The research was supported by  the internal research project R-AGR-0500 of the University of Luxembourg. The first author was partially supported by the Hungarian Scientific Research Fund (OTKA) K104178. The second author was partially supported by the Hungarian Scientific Research Fund (OTKA) K115799 }



%

\date{\today}


\maketitle

\begin{abstract}
An $n$-ary associative function is called reducible if it can be
written as a composition of a binary associative function. We
summarize known results when the function is defined on a chain and
is nondecreasing. Our main result shows that associative idempotent
and nondecreasing functions are uniquely reducible.
\end{abstract}

\section{Introduction}\label{s1}
In this paper we investigate the class of functions $F:X^n\to X$
($n\ge 2$) defined on a chain (i.e., totally ordered set) $X$ that
are nondecreasing, idempotent and associative. For arbitrary set $X$
the study of associativity stemmed back to the pioneering work of
D\"ornte \cite{Dor28} and Post \cite{Pos40}. Dudek and Mukhin
\cite{DudMuk96, DM} gave a characterization of reducibility using
the terminology of a neutral element (see Theorem \ref{thmDM}).
While their result is essential from a theoretic point of view, it
is not easy to apply it for a given situation unless the function
originally has a neutral element (for further details see also
\cite{GG2016}). Ackerman \cite{A} made a complete characterization
of quasitrivial associative functions. In his paper it was shown
that every quasitrivial associative function is derived from a
binary or a ternary function.

Couceiro and Marichal showed in \cite{CM2012} that continuous
symmetric cancellative and associative $n$-ary functions defined on
a nonempty real interval are reducible (see Remark 4 of
\cite{CM2012}). Although they established reducibility under some
hypotheses that are not related to those of the present paper, it
also shows that reducibility is an important property in the study
of associative $n$-ary functions. Reducibility and
extremality\footnote{The definition of extremality stems from
\cite{Szekely2014}.} of quasitrivial associative symmetric
nondecreasing functions were studied in \cite{DKM}.

The paper is organized as follows. Section \ref{s2} contains the
basic definitions and notation. In Section \ref{s3}.1 we collect the
preliminary results in the case when $F:X^n\to X$ is idempotent,
monotone, associative and has a neutral element. This part is based
on \cite{DM} and \cite{GG2016}.
In Section \ref{s3}.2  we complete the study of reducibility of
quasitrivial nondecreasing associative $n$-ary functions (without
the assumption of symmetry). In Section \ref{s4} we present the main
results about the reducibility of idempotent, nondecreasing,
associative functions. Because of its simplicity we present the
symmetric case with useful lemmas (see Lemma \ref{l1} and \ref{l2})
in Section \ref{s4.1}. In Section \ref{s4.2} we prove the general
result. The main technicality is that we have to divide the proof
into two subcases. Theorem \ref{t1} can be used only for $n=3$, and
another inductive proof (Theorem \ref{t2}) works for $n>3$. In
Section \ref{s5} we discuss extremality which holds in many special
cases but not for every associative idempotent nondecreasing
function.
We also and monotonicity as a relaxation of the property of the
nondecreasingness.

\section{Definitions and notation}\label{s2}
Let $X$ be an arbitrary set and $F:X^n\to X$ an $n$-ary function.
We denote by $S_n$ the symmetric group on the set $\{1, \dots, n\}$.
Now we give a sequence of definitions:
\begin{de}
The function $F:X^n\to X$ is called
\begin{enumerate}[(i)]
\item {\it idempotent} if $F(x,\dots,x)=x$ for every $x\in X$,
\item  {\it symmetric} if $F(x_1, \dots,x_n)=F(x_{\sigma(1)},\dots, x_{\sigma(n)})$ for every $x_1,\dots, x_n \in X$ and every permutation $\sigma\in S_n$,
\item {\it quasitrivial} (or {\it conservative}) if for any $x_1,\dots, x_n\in X$ $$F(x_1, \dots,x_n)\in\{x_1, \dots, x_n\},$$
\item $n$-{\it associative} if for every $x_1, \dots, x_{2n-1}\in X$ and $1
    \le i \le n-1$ we have
    \begin{equation}\label{eqnas}
    \begin{split}
    &F(F(x_1, \dots, x_n),x_{n+1}, \dots, x_{2n-1})= \\ & F(x_1, \dots, x_i
    ,F(x_{i+1}, \dots, x_{i+n}), x_{i+n+1}, \dots, x_{2n-1}).
    \end{split}
    \end{equation}
\end{enumerate}
\end{de}
We usually say that $F:X^n\to X$ is associative and we only write
that $F$ is $n$-associative if we would like to emphasize the number
of variables in $F$.

We say that $F:X^n\to X$ has a {\it neutral element} denoted by
$e\in X$ if for every $x \in X$ and $1 \le i\le n$ we have $F(e,
\dots ,e, x, e, \dots, e) =x$, where $x$ is in the $i$'th coordinate
of $F$.

For any integer $k \ge 0 $  and any $x \in X$, we set $k \cdot x =
\underbrace{x, \dots, x}_{k ~\text{times}}$. For instance,
idempotency of $F$ can be written in the form $F(n\cdot x)=x$.

From now on, $X$ will be a totally ordered set. For any $n\in
\mathbb{N}$ the function $F:X^n\to X$ is called {\it nondecreasing}
(resp. {\it nonincreasing}) if
\begin{equation}\label{eqFF}F(a_1, \dots, a_n)\ge F(b_1,\dots, b_n) ~ (\textrm{resp. } F(a_1, \dots, a_n)\le F(b_1,\dots, b_n)),\end{equation}
for every pair of $n$-tuples $(a_1, \dots, a_n),(b_1,\dots, b_n)$,
where $a_i\ge b_i\in X$ for $1 \le i \le n$.

The function $F$ is called \emph{monotone in the $i$-th variable} if
for all fixed elements $a_1\stb a_{i-1}, a_{i+1} \stb a_n$ of $X$,
the $1$-ary function defined as $$f_i(x):=F(a_1 \stb
a_{i-1},x,a_{i+1} \stb a_n)$$ is nondecreasing or nonincreasing. The
function $F$ is called \emph{monotone} if it is monotone in each of
its variables.

We use the lattice notation for the minimum ($\wedge$) and for
maximum ($\vee$) of a set. Hence we introduce the notation
$$\wedge_{i=1}^n x_i=\min\{x_1, \dots, x_n\},$$
$$\vee_{i=1}^n x_i=\max\{x_1, \dots, x_n\}.$$

\section{Preliminary results}\label{s3}

\begin{de}

We say that $F:X^n\to X$ {\it is derived from} $G:X^2\to X$ if $F$
can be written in the form $$F(x_1, \dots, x_n)=x_1\circ \dots \circ
x_n,$$ where $x\circ y=G(x,y)$. We note that this expression is
well-defined for $n \ge 3$ if and only if $G$ is associative. If
such a $G$ exists, then we say that $F$ is {\it reducible}.
\end{de}

We note that if $n=2$ then the function $F$ is derived from itself.

The previous definition only deals with the existence of a binary
function from which a given $n$-associative function can be derived.
The uniqueness of the binary function follows from certain
conditions.
The following result was proved first in \cite[Proposition
3.5]{DKM}.
\begin{proposition}
Assume that the function $F : X^n\to X$ is associative and
derived from an associative idempotent binary function. 
Then the binary function is unique.
\end{proposition}

In our case, when $X$ is a totally ordered set and $F$ is monotone,
we can strengthen the previous statement. The result
presented here follows from \cite[Lemma 3.4]{GG2016} 
when $F$ is chosen to be monotone.
\begin{proposition}\label{pGG2}
Let $X$ be a totally ordered set and $F:X^n\to X$ an associative
idempotent monotone function, which is derived from an associative
binary function $G$. Then $G$ is idempotent as well.
\end{proposition}

Combining the previous statements we get:

\begin{corollary}\label{corU}
Let $X$ be a totally ordered set. Let $F$ be an associative
idempotent monotone function which derived from a binary function
$G:X^2\to X$,
then 
$G$ is uniquely determined by $F$.
\end{corollary}

\subsection{Neutral element}\label{s3.2}
Suppose that $F:X^n \to X$ is an associative function having the
neutral element $e\in X$. Then one can define $G:X^2\to X$ by
\begin{equation}\label{eqf2}
 G(a,b)=F(a,(n-2)\cdot e, b)\end{equation}
for every $a,b\in X$. The following theorem of Dudek and Mukhin
\cite{DM} shows a general result for an arbitrary set $X$.

\begin{theorem}\label{thmDM}
Let $X$ be a nonempty set. Let $F: X^n\to X$ be an associative
function. Then $F$ is derived from a binary function $G$ if and only
if $F$ has a neutral element or one can adjoin\footnote{Adjoining an
element to an $n$-associative function $F$ means to define an
$n$-associative function $\bar{F}$ defined on $X \cup \{ e\} $,
where $e$ is a neutral element for $\bar{F}$ such that $\bar{F}(x_1
, \ldots , x_n)=F(x_1 , \ldots , x_n)$ for every $x_1 , \ldots , x_n
\in X$.} a neutral element to $X$ for $F$. In this case such a $G$
can be defined by \eqref{eqf2}.
\end{theorem}
We note that the previous statement also holds for $n=2$. Indeed, every associative binary function is reducible and if an associative function $F$ has no neutral element, then we can adjoin one. Let $e\not\in X$ be an element and let $\bar{F}$ be defined as $\bar{F}(x,y)=F(x,y)$ for $x,y\in X$ and $\bar{F}(z,e)=\bar{F}(e,z)=z$ for every $z\in X\cup \{e\}$. It is easy to check that $\bar{F}$ is associative on $X\cup \{e\}$. 

The following statement was proved  in \cite[Proposition
3.13]{GG2016} applying the previous structural theorem.
 \begin{proposition}\label{pGG}
    Let $X$ be a totally ordered set and $F:X^n\to X$ an associative monotone 
    idempotent function with a neutral element $e$.
    Let $G$ be defined by \eqref{eqf2}. \\
    Then $F$ is derived from the binary function
    $G$, which is also associative, idempotent, monotone and has the same neutral element $e$.
 \end{proposition}
Since every monotone, idempotent associative binary function is
nondecreasing by \cite[Lemma 3.10]{GG2016}, the previous statement
immediately has a simple consequence.
\begin{corollary}\label{cGG}
Let $X$ and $F$ as in Proposition \ref{pGG}. Then $F$ is
nondecreasing.
\end{corollary}
\begin{observation}
Let $X$ and $F$ as in Proposition \ref{pGG}. If $F$ is symmetric,
then $G$ defined by \eqref{eqf2} is also symmetric.
\end{observation}

Lemma \ref{lemmaNeQt} shows a connection between the existence of a
neutral element and quasitriviality. The base of the idea appears in
the Czoga\l a-Drewniak's theorem \cite{Czogala1984} where
$X=[0,1]$. For the sake of completeness we present a short proof here. 

\begin{lemma}\label{lemmaNeQt}
Let $X$ be a totally ordered set and $F:X^n\to X$ an associative,
idempotent, monotone function having a neutral element $e$. Then $F$
is quasitrivial.
\end{lemma}

\begin{proof}
By Corollary \ref{cGG}, we can automatically assume that $F$ is
nondecreasing.
For $n=2$ and $x,y\in X$, we distinguish two different cases:
\begin{enumerate}
\item ($x\le e,y\le e$) or  ($e\le x, e\le y$),
\item ($x\le e\le y$) or ($y\le e\le x$).
\end{enumerate}
We show that in each case $F(x,y)$ is either the maximum or the
minimum, thus it is quasitrivial. In Case 1 if $x\le e,y\le e$, then
by the nondecreasingness of $F$ we get  \begin{align*}\label{eqal1}
    x=F(x,e)&\ge F(x,y)\\
    y=F(e,y)&\ge F(x,y).
\end{align*}
Thus $x\wedge y\ge F(x,y)$.

On the other hand if $x\le y$ (the case $x\ge y$ can be handled
similarly), then
$$x=F(x,x)\le F(x,y)\le F(y,y)=y,$$
by monotonicity and idempotency. This implies that $F(x,y)=x\wedge
y$.

Similarly if $e\le x, e\le y$, it can be obtained that $F(x,y)=x\vee
y$.

In Case 2 the two subcases can be handled similarly. Now we deal
with $x\le e\le y$. we denote $F(x,y)=\theta$. Assume that
$x\le\theta\le e \le y$, then using associativity, we get
\begin{equation}
    F(x,\theta)=F(x,F(x,y))=F(F(x,x), y)=F(x,y)=\theta
\end{equation}
On the other hand, since $x\le e, \theta\le e$, we have already
proved that $$F(x,\theta)=x\wedge\theta=x.$$ This shows that
$\theta=x$. For  $x\le e\le \theta \le y$ similarly we have
$$\theta=F(\theta,y)=y.$$
Thus we get that the binary function $F$ is quasitrivial.

If $n>2$ and $F$ is an $n$-associative idempotent non-decreasing and
have a neutral element, then we can use Proposition \ref{pGG}. Thus
there exists a binary function $G$ which is associative, idempotent,
non-decreasing and have a neutral element. By the case $n=2$ we know
that $G$ is quasitrivial and, since $F$ is derived from $G$, $F$ is
also quasitrivial.
\end{proof}

\subsection{Quasitriviality}\label{s3.3}
In \cite[Theorem 3.3 and Corollary 3.4]{DKM}  the authors proved the
following characterization for quasitrivial symmetric nondecreasing
and associative functions.
\begin{theorem}\label{propQS}
Let $X$ be a totally ordered set and let $F\colon X^n\to X$ be a
quasitrivial symmetric nondecreasing associative function. Then $F$
is reducible. More precisely, $F$ is derived from $G\colon X^2\to X$
defined by \begin{equation}\label{eqxy} G(x,y)=F((n-1)\cdot
x,y)=F(x,(n-1)\cdot y).
\end{equation}
\end{theorem}
It is easy to see that function $G$ defined by \eqref{eqxy} is
quasitrivial,    symmetric and nondecreasing. In \cite[Theorem
3.3]{DKM} it was also proved that in this case
\begin{equation}\label{eqext1}
F(x_1, \dots, x_n)=G(\wedge_{i=1}^n x_i, \vee_{i=1}^n x_i).
\end{equation}
This means that $F$ is extremal (see Definition \ref{deext}).

One can prove that $F$ remains reducible if we eliminate the
symmetry condition of $F$. The result is weaker in the sense that it
only shows the existence of such a decomposition (see Theorem
\ref{thmQA}). We note that the analogue of \eqref{eqext1} does not hold (for further details see Section \ref{SecExt}). 

The following result is an easy consequence of \cite[Theorem 1.4]{A}
using the statement therein for $A_2=\emptyset$.

\begin{theorem}\label{thmAkk}
Let $X$ be an arbitrary set. Suppose $F:X^n\to X$ is a quasitrivial
$n$-associative function. If $F$ is not derived from a binary
function, then $n$ is odd and there exist $b_1 , b_2$ $(b_1\ne b_2)$
such that for any $a_1 \stb a_n \in \{b_1, b_2\}$
\begin{equation}\label{eqbk}
    F(a_1, \dots, a_n)=b_i~~(i=\{1,2\}),
\end{equation} where $b_i$ occurs odd number of times. 
\end{theorem}
As a consequence of this theorem we prove the following:
\begin{theorem}\label{thmQA}
Let $X$ be a totally ordered set and let $F\colon X^n\to X$ be an
associative quasitrivial nondecreasing function. Then $F$ is
reducible.
\end{theorem}
\proof By contradiction we assume that $F$ is not derived from a
binary function. Now we apply the previous theorem since we intend
to show that in this case the conditions for $b_1, b_2$ cannot be
satisfied. Thus every associative, quasitrivial, nondecreasing
function defined on a totally ordered set $X$ is reducible.

According to Theorem \ref{thmAkk}, if $F$ is not reducible, then $n$
is odd. Hence $n\ge 3$ and there exist $b_1, b_2$ satisfying
equation \eqref{eqbk}.
Since $b_1\ne b_2$, 
we may assume that $b_1<b_2$ (the case $b_2<b_1$ can be handled
similarly). By our assumption on $b_1$ and $b_2$ we have
\begin{equation}\label{eq:b1b2}
F(n \cdot b_1)=b_1, ~ F(b_2, (n-1) \cdot b_1)=b_2, ~ F(2 \cdot b_2,
(n-2)\cdot b_1)=b_1.
\end{equation}
Since $F$ is nondecreasing we have $$ F(n \cdot b_1) \le F(b_2,
(n-1) \cdot b_1)\le F(2\cdot b_2, (n-2)\cdot b_1).$$ This implies
$b_1=b_2$, a contradiction. \qed

\section{Main results}\label{s4}
In this section we prove that every associative idempotent
nondecreasing function defined on a totally ordered set $X$ is
derived from a binary function $G$. As it was shown in Corollary
\ref{corU}, $G$ is also unique. This result generalizes some of the
previous results on reducibility. As a consequence of Theorem
\ref{thmDM} this means that if an associative idempotent
nondecreasing function $F$ defined on a totally ordered set $X$,
then either there is a neutral element of $F$ or we can adjoin an
element to $X$ which acts as a neutral element of $F$. We note that
all of our statements also hold for $n=2$ but bring no information
in this case. Practically, we just deal with the cases when $n\ge
3$.

\subsection{Symmetric case}\label{s4.1}
The symmetric case (as usual) is much simpler than the general one
but we present a separate argument here. Our result is based on the
following two lemmas.
\begin{lemma}\label{l1}
Let $X$ be a totally ordered set and $F:X^n\to X$ an associative
nondecreasing idempotent function. Then for every $a,c\in X$
$$F(a, (n-1)\cdot c)=F((n-1)\cdot a, c).$$
\end{lemma}

\begin{proof}
If $a=c$, then the statement trivially follows from the idempotency of $F$. 
We assume that $a<c$. (The case $a>c$ can be handled similarly.) We
denote $F((n-1)\cdot a, c)$ by $\theta$. Since $F$ is nondecreasing
and idempotent we have $a \le \theta \le c$.
\begin{equation*}
\label{eqge}
\begin{split}
& \theta=F((n-1)\cdot a, c)\le F(a, (n-1)\cdot  c)\le F(\theta, (n-1)\cdot  c)=\\
& F(F((n-1)\cdot a, c), (n-1)\cdot  c)=F((n-1)\cdot a, F(n\cdot c))= \\
& F((n-1)\cdot a, c)= \theta.
\end{split}
\end{equation*}
Thus, we get $F(a,(n-1)c)=F((n-1)a, c)$.
\end{proof}

\begin{remark}\label{Rem:ac}
As a consequence of the previous lemma we obtain that if $F$ is an
associative idempotent nondecreasing function, then $F(k \cdot a,
(n-k) \cdot c)$ is the same for every $1 \le k \le n-1$. Indeed, if
$a\le c$, then $F((n-1) \cdot a,  c) \le F(k \cdot a, (n-k) \cdot c)
\le F(a, (n-1) \cdot c)$. If $a \ge c$, then  $F((n-1) \cdot a,  c)
\ge F(k \cdot a, (n-k) \cdot c) \ge F(a, (n-1) \cdot c).$
\end{remark}

\begin{lemma}\label{l2}
Let $X$ be a totally ordered set and $F:X^n\to X$ an associative
idempotent and nondecreasing function. Then the  function $G$
defined by \begin{equation}\label{eqG}
 F(a,(n-1)\cdot c)= F((n-1)\cdot a,c)=G(a,c).
 \end{equation}
 is associative idempotent and nondecreasing.
\end{lemma}
We note that by Lemma \ref{l1} and Remark \ref{Rem:ac}, $G$ is
well-defined and $F(a,(n-1)\cdot c)=F(k \cdot a,(n-k)\cdot c) $ for
every $k=1 \stb n-1$.
\begin{proof}
 It is clear that $G$ is idempotent and nondecreasing.
 The following equation shows that $G$ is associative.
 \begin{equation*}\label{eqass}
 \begin{split}
 &G(a,G(b,c))=F((n-1) \cdot a,F(b,(n-1)\cdot c)=\\
 &F(F((n-1) \cdot a,b),(n-1)\cdot c)=G(G(a,b),c).
 \end{split}
 \end{equation*}
\end{proof}
 Now we investigate the question of reducibility for the symmetric case.
 \begin{theorem}\label{thmANIS}
 Let $X$ be a totally ordered set and let $F:X^n\to X$ be an associative symmetric nondecreasing idempotent function.
 Then $F$ is derived from a unique binary function $G:X^2\to X$ which can be obtained as
\begin{equation}
 G(a,c)=F(a,(n-1)\cdot c).
 \end{equation}
  Moreover
  \begin{equation}\label{eqextsym}
 F(x_1, \dots, x_n)=G(\wedge_{i=1}^{n} x_i,  \vee_{i=1}^{n} x_i).
 \end{equation}
\end{theorem}

 \begin{remark}\label{re:sym}
 Equation \eqref{eqextsym} means that $F$ is extremal (see Section \ref{SecExt}).
 \end{remark}
 \begin{proof}
 Applying Lemma \ref{l1},
 we can define $G$ for any $a,c\in X$ by
 $$G(a,c)=F((n-1)\cdot a,c)=F(a,(n-1)\cdot c).$$
The uniqueness of the binary function follows from Corollary
\ref{corU} so we only have to verify that $G$  fulfils our
requirements.

Since $F$ is nondecreasing, we have that
\begin{equation}\label{eqxi}G(a,c)=F((n-1)\cdot a,c)\le F(a, x_1,\dots, x_{n-2},c)\le F(a,(n-1)\cdot c) =G(a,c) \end{equation}
for every $a\le x_1,\dots, x_{n-2}\le c$. We get that the
inequalities in \eqref{eqxi} are equalities. Thus by the symmetry of
$F$, the value of $F(x_1 \stb x_n)$ depends only on
$\wedge_{i=1}^{n} x_i $ and $\vee_{i=1}^{n} x_i$.

Using the symmetry of $F$ we can reorder the entries of $F$ and
 we get
 \begin{equation*}
 F(x_1, \dots, x_n)=F(\wedge_{i=1}^{n} x_i,\dots, \vee_{i=1}^{n} x_i)=G(\wedge_{i=1}^{n} x_i,  \vee_{i=1}^{n} x_i).
 \end{equation*}
 This argument shows that $F$ is derived from $G$ (and extremal).
 \end{proof}
 \subsection{General case}\label{s4.2}
In this section we do not assume that our functions are symmetric.
In Theorem \ref{t1} and \ref{t2} we prove the reducibility of
associative idempotent nondecreasing $n$-ary functions for $n\ge3$
which is the main result of this section. It seems from our argument
that the cases $n=3$ and $n\ge 4$ should be handled in different
ways and separately. First we discuss the case $n=3$.
 \begin{theorem}\label{t1}
 Let $X$ be a totally ordered set and let $F:X^3\to X$ be an associative idempotent nondecreasing function. Then $F$ is derived from a unique binary function denoted by $G:X^2\to X$.
 The function $G$ can be defined by \begin{equation}\label{eqG3}
 G(a,c)=F(a,c,c)= F(a,a,c).
 \end{equation}
 \end{theorem}
 \begin{proof}
 By Lemma \ref{l1}, $G$ can be defined by \eqref{eqG3}. Applying Lemma \ref{l2} we get that $G$ is associative nondecreasing and idempotent.
 We need to show that \begin{equation*}
 F(a,b,c)=G(a,G(b,c))=G(G(a,b),c)
 \end{equation*}
 for every $a,b,c\in X$.

 If $a\le b\le c$ (the case $a\ge b \ge c$ can be handled similarly), then we can directly apply \eqref{eqG3} and we obtain $$G(a,c)=F(a,a,c)\le F(a,b,c)\le F(a,c,c)=G(a,c).$$

 On the other hand, since $G$ is nondecreasing and idempotent, we have
 \begin{equation}\label{eqgac}
 \begin{split}
 &G(a,G(b,c))\le G(a,G(c,c))=G(a,c),\\
 &G(G(a,b),c)\ge G(G(a,a),c))=G(a,c).
 \end{split}
 \end{equation}
By the associativity of $G$ and equation \eqref{eqgac} we get
$G(a,c) \le G(G(a,b),c)=G(a,G(b,c)) \le G(a,c)$. Hence
$F(a,b,c)=G(a,c)= G(G(a,b),c)=G(a,G(b,c))$ as required.

Assume ($a\le b, ~c\le b$) or  ($a\ge b, ~c\ge b $) (i.e., $b$ is
the smallest or the largest among $a,b,c$). We could assume that all
of the previous relations are strict inequalities. Otherwise we are
in the previous case but the proof works for these cases as well.

We introduce the following notation
\begin{equation*}
\begin{split}
 &\theta_1=G(a,b)=F(a,a,b)=F(a,b,b),\\
 &\theta_2=G(b,c)=F(b,b,c)=F(b,c,c).
\end{split}
\end{equation*}
  Then we get
 \begin{equation}\label{eqF3.1}
 \begin{split}
 &F(a,b,c)=F(F(3\cdot a),F(3\cdot b), c)=\\
 &F(a, F(a,a,b),F(b,b,c))=F(a,\theta_1,\theta_2).
 \end{split}
 \end{equation}
and
 \begin{equation}\label{eqF3.2}
 \begin{split}
 &F(a,b,c)=F(a,F(3\cdot b),F(3\cdot c))=\\
 &F(F(a, b, b), F(b, c, c),c)=F(\theta_1, \theta_2,c).
 \end{split}
 \end{equation}

 Suppose that $b=\max \{a,b,c\}$ ($b=\min\{a,b,c\}$ can be handled similarly).
 If $\theta_1\le \theta_2$, then $a\le b$ implies $G(a,a)=a\le G(a,b)= \theta_1\le \theta_2$,
 so $a, \theta_1, \theta_2$ are in increasing order.
 Therefore by the previous case
 \begin{equation*}
 F(a, \theta_1, \theta_2)=G(a,\theta_2)=G(a, G(b,c)).
 \end{equation*}
 Using equation \eqref{eqF3.1} we obtain that $F(a,b,c)=G(a, G(b,c))$, which equals to $G(G(a,b),c)$ since $G$ is associative.

 If $\theta_1\ge \theta_2$, then by $c\le b$ we get that
 $c=G(c,c)\le G(b,c) =\theta_2\le \theta_1$. Now the sequence $\theta_1, \theta_2,c$ is in decreasing order, hence
 \begin{equation*}
 F(\theta_1, \theta_2,c)=G(\theta_1,c)=G(G(a,b),c).
 \end{equation*}
 Using equation \eqref{eqF3.2} we get that $F(a,b,c)=G(G(a,b),c)$.
 Finally, the associativity of $G$ gives the result, finishing the proof of Theorem \ref{t1}.
 \end{proof}

 Now we prove the analogous result for $n\ge 4$. The main problem is that in case $n=3$ we heavily use the fact that every ordered triple $(a,b,c)$ is either monotone (i.e., $a \le b \le c$ or $a \ge b \ge c$) or one of its extrema is in the middle (i.e., $a ,c \le b$ or $b \le a,c$). Generally, for $n>3$ there are plenty  other cases.  Therefore we follow another way to generalize the previous result.
 We start with two lemmas.
 \begin{lemma}\label{l4}
 Let $X$ be a totally ordered set and $F:X^n\to X$ an associative idempotent nondecreasing function. Then
 \begin{equation}\label{eqFn}
 F(x_1,\dots,x_{i-1},2\cdot x_i, x_{i+1}, \dots, x_{n-1})= F(x_1,\dots, x_i, 2\cdot  x_{i+1}, x_{i+2}, \dots, x_{n-1})
 \end{equation}
 holds for every $i\in \{1,\dots, n-2\}$ and $x_1,\dots,x_{n-1}\in X$.
\end{lemma}

\begin{proof}
Lemma \ref{l1} gives  $F((n-1)\cdot a,c)=F(a,(n-1)\cdot c)$. Since
$F$ is nondecreasing we obtain
\begin{equation}\label{eq17}
F((n-1)\cdot a,c)=F(k\cdot a,(n-k)\cdot c)
\end{equation}
for every $1 \le k \le n-1$ (as in Remark \ref{Rem:ac}).
The following direct calculation proves the statement. We use the
idempotency of $F$ in the first and last equalities, the
associativity of $F$ and in the second and fourth equalities and we
use equation \eqref{eq17} for $x_i$ and $x_{i+1}$ in the third
equality
\begin{align*}
&F(x_1,\dots, 2\cdot x_i, x_{i+1}, \dots, x_{n-1}))=F(x_1,\dots, x_i, F(n\cdot x_i), x_{i+1}, \dots, x_{n-1}))=\\
&F(x_1, \dots, 2\cdot x_i, F((n-1)\cdot x_i, x_{i+1}),\dots, x_{n-1})=\\
&F(x_1, \dots, 2\cdot x_i, F((n-2)\cdot x_i,2\cdot x_{i+1}),\dots, x_{n-1})=\\
&F(x_1,\dots, F(n\cdot x_i), 2\cdot x_{i+1}, \dots, x_{n-1})=
F(x_1,\dots, x_i, 2\cdot x_{i+1}, \dots, x_{n-1}).
\end{align*}
\end{proof}

\begin{corollary}\label{cor n-1} Let $X$ and $F$ be as above.
One can define $H:X^{n-1}\to X$ by the following formula
\begin{equation}\label{eq117}
 H(x_1,\dots, x_{n-1})=
 F(2\cdot x_1, x_2,\dots,  x_{n-1})=\ldots=
 F(x_1,\dots, x_{n-2},2\cdot x_{n-1})
\end{equation}
\end{corollary}
\begin{remark}\label{rem:Hidde}
We note that 
$H$ defined by \eqref{eq117} also is idempotent and nondecreasing if
$F$ has the same properties.
\end{remark}

\begin{lemma}\label{l5}
Let $X$ be a totally ordered set and $F:X^n\to X$ an associative
idempotent nondecreasing function. Then $H:X^{n-1}\to X$ which is
defined in Corollary \ref{cor n-1} is associative.
\end{lemma}

\begin{proof} 
The following equations hold for any $k\in \{3,\dots ,n-1\}$ 
\begin{align*}
&H(x_1,\dots, x_{k-1}, H(y_1, \dots, y_{n-1}),x_{k+1}, \dots, x_{n-1})=\\
&F(2\cdot x_1, \dots, x_{k-1}, F(2 \cdot y_1, \dots, y_{n-1}), x_{k+1}, \dots, x_{n-1})=\\
&F(2\cdot x_1, \dots, x_{k-2},F(x_{k-1}, 2\cdot y_1, \dots, y_{n-2}), y_{n-1}, x_{k+1},\dots, x_{n-1})=\\
&H(x_1, \dots, x_{k-2},H(x_{k-1}, y_1, \dots, y_{n-2}),  y_{n-1},
x_{k+1}\dots, x_{n-1}).
\end{align*}
For $k=2$ the previous calculation does not hold. In that case we
get the following equation using \eqref{eq117}.
\begin{align*}
&H(x_1, H(y_1, \dots, y_{n-1}),x_3, \dots, x_{n-1})=\\
&F(x_1, F( 2\cdot y_1, \dots, y_{n-1}),  x_3, \dots, 2\cdot x_{n-1})=\\
&F(F(x_1, 2 \cdot y_1, \dots, y_{n-2}), y_{n-1}, x_3, \dots, 2 \cdot x_{n-1})=\\
&H(H(x_1, y_1, \dots, y_{n-2}),  y_{n-1}, x_3\dots, x_{n-1}).
\end{align*}
\end{proof}
Since $H:X^{n-1}\to X$ is associative idempotent and nondecreasing,
we can use induction for $n\ge 3$.

\begin{theorem}\label{t2}
Let $X$ be a totally ordered set and let $F:X^n\to X$ $(n\ge 2)$ be
an associative idempotent nondecreasing function. Then there exists
a unique associative idempotent nondecreasing binary function
$G:X^2\to X$ from which $F$ is derived. Moreover, $G$ can be defined
by
\begin{equation}\label{eq1170}
G(a,c)=F(a, (n-1)\cdot c)=F((n-1)\cdot a, c).\end{equation}
\end{theorem}


\begin{proof} For $n=2$ the statement is automatically true.
The statement is proved by induction for $n\ge 3$. Theorem \ref{t1}
gives the result for $n=3$.

Now we assume that $n>3$. By Lemma \ref{l1} and Lemma \ref{l2},
$G:X^2\to X$ a is well-defined associative idempotent nondecreasing
function. Let $H:X^{n-1}\to X$ be defined by \eqref{eq117} as in
Corollary \ref{cor n-1}.
 The function $H$ is associative nondecreasing and idempotent according to Lemma \ref{l5}.

Now we recall the notation $G(a,b)=a \circ b$ which is well-defined
since $G$ is associative by Lemma \ref{l2}.

 By induction, $H$ is derived from a binary function.
Since
\begin{equation}\label{eq1177} a \circ b=G(a,b)=F((n-1)\cdot
a,b)=H((n-2)a, b)\end{equation} we have that $H$ is derived from
$G$, i.e: \begin{equation}\label{eq1717} H(x_1,x_2,\dots,
x_{n-1})=x_1\circ x_2\circ \cdots\circ x_{n-1}.
\end{equation}

Now we show that $F$ is also derived from the same binary function
$G$.
 \begin{equation}\label{eq1777}
\begin{split}
&F(x_1, x_2 \dots, x_n)=F(F(n\cdot x_1), x_2, \dots, x_n)=\\
&F((n-2)\cdot x_1, F(2 \cdot x_1, x_2,\dots ,x_{n-1}), x_n)=\\
&H((n-3)\cdot x_1, H(x_1, x_2, \dots, x_{n-1}), x_n )=\\
&x_1\circ \ldots \circ x_1 \circ (x_1 \circ x_2\circ \ldots  \circ x_{n-1})\circ x_n=\\
&x_1\circ x_2\circ \ldots \circ x_{n-1}\circ x_n.
\end{split}
\end{equation}
In the second equation we use the associativity of $F$, in the third
we substitute $H$ using that $n-2\ge 2$, in the fourth equation we
apply \eqref{eq1717}, in the last equation we use the idempotency
and associativity of $G$. Equation \eqref{eq1777} shows that $F$ is
also derived from $G$. By \eqref{eq1177}, $G$ is of the form
\eqref{eq1170}. The uniqueness of $G$ comes from Corollary
\ref{corU}.

\end{proof}

\begin{corollary}\label{cormain}
Let $X$ be a totally ordered set and $n\ge 2$ an integer. An
associative idempotent monotone function $F:X^n\to X$ is reducible
if and only if $F$ is nondecreasing.
\end{corollary}
\begin{proof}
($\Longleftarrow$): This immediately follows from \cite[Corollary
3.12]{GG2016} which states that if $F:X^n\to X$ ($n\ge2$) is
associative idempotent monotone (at least in the first and the last
variables) and reducible, then $F$ is nondecreasing (in each of its
variables).


($\Longrightarrow$): By Theorem \ref{t2}, every associative
idempotent nondecreasing $n$-ary function ($n\ge 2$) is reducible.
\end{proof}
\begin{example}\label{Ex1}
Let $X$ be a totally ordered  Abelian group with respect to the
addition and let $g:X\to X$ be a monotone bijective function on $X$.
Then the function $$F(x,y,z)=g^{-1}(g(x)-g(y)+g(z))$$ is idempotent
associative monotone but nondecreasing. Thus $F$ is not reducible.
\end{example}

\section{Further remarks}\label{s5}

\subsection{Extremality}\label{SecExt}
\begin{de}\label{deext}
We say that $F:X^n\to X$ is {\it extremal}\footnote{In
\cite{Szekely2014} a mean $\mu: (\bigcup_{n\in
\mathbb{N}}\mathbb{R}^{n})\to \mathbb{R}$ was called {\it extremal}
if for all elements $a_1\le a_2 \le \dots \le a_n\in \mathbb{R}$ we
have $\mu(a_1, a_2, \dots, a_n)=\mu(a_1, a_n)$.} if there exists a
$G:X^2\to X$ such that for every $x_1 \stb x_n \in X$ we have that
$F(x_1, \dots, x_n)$ equals to either $G(\wedge_{i=1}^n x_i,
\vee_{i=1}^n x_i)$ or $G(\vee_{i=1}^n x_i, \wedge_{i=1}^n x_i)$. In
particular, if $F : X^n \to X$ is symmetric and ext\-remal, then
there exists a symmetric $G : X^2\to X$ such that  $F(x_1\stb
x_n)=G(\wedge_{i=1}^n x_i, \vee_{i=1}^n x_i)$.
\end{de}

In \cite{DKM} it was shown (as we have already stated in equation
\eqref{eqext1}) that if $F:X^n\to X$ is associative, quasitrivial,
symmetric and nondecreasing defined on the chain $X$ then $F$ is
extremal. As a possible generalization it was shown in Theorem
\ref{thmANIS} that instead of quasitriviality it is enough to assume
idempotency (see also Remark \ref{re:sym}). Namely:
\begin{proposition}
Let $X$ be a totally ordered set.
Then every associative symmetric nondecreasing idempotent function $F:X^n\to X$ is extremal. 
\end{proposition}

In \cite[Theorem 2.6.]{GG2016}, it was shown that every associative nondecreasing idempotent function having a neutral element is extremal. 


If $F:X^n\to X$ is associative quasitrivial and nondecreasing, then
$F$ is not necessarily extremal. It can be shown easily that the
projection to the $i$'th coordinate is not extremal. This also gives
an example for associative idempotent nondecreasing function, which
is not extremal.






\subsection{Monotonicity}
Although in the binary case it cannot happen, Example \ref{Ex1}
shows that there exists an associative idempotent monotone function,
which is not nondecreasing (so it is not reducible by Corollary
\ref{cormain}). The characterization of these functions are not
known yet. We conjecture the following (in the spirit of Acz\'elian
$n$-ary semigroups \cite{CM2012}):
\begin{conj}\label{conj1}
 Let $X$ be a totally ordered Abelian group with respect to the addition.
 An associative idempotent strictly\footnote{A function is strictly monotone if and only if the function is monotone and every inequality in the definition of monotonicity (see equation \eqref{eqFF}) is strict.} monotone function  $F:X^n\to X$ is not reducible if and only if $n$ is odd and there exists a monotone bijection $g:X\to X$ such that
 \begin{equation}
     F(x_1,x_2,\dots, x_n)=g^{-1}(\sum_{i=1}^n(-1)^ig(x_i)).
     \end{equation}
\end{conj}

The 'if' part of the statement is clear.
We note that if Conjecture \ref{conj1} holds for $X=\mathbb{R}$,
then such an $F$ must be automatically continuous, since every
monotone bijection on an interval is continuous.

    \section*{Acknowledgement}
    The authors would like to thank the referee for the valuable comments and suggestions which improved the quality of this paper.


\end{document}